\documentclass[12pt]{amsart}
\usepackage[active]{srcltx}
\usepackage{calc,amssymb,amsthm,amsmath,amscd, eucal,ulem,mathrsfs}
\usepackage{alltt}
\usepackage[left=1.35in,top=1.25in,right=1.35in,bottom=1.25in]{geometry}
\usepackage{mathtools}
\usepackage{soul}
\usepackage[monochrome,dvipsnames]{xcolor}
\usepackage{hyperref}

\renewcommand{\geqslant}{\geq}
\renewcommand{\leqslant}{\leq}

\makeatletter
\def\@tocline#1#2#3#4#5#6#7{\relax
  \ifnum #1>\c@tocdepth %
  \else
    \par \addpenalty\@secpenalty\addvspace{#2}%
    \begingroup \hyphenpenalty\@M
    \@ifempty{#4}{%
      \@tempdima\csname r@tocindent\number#1\endcsname\relax
    }{%
      \@tempdima#4\relax
    }%
    \parindent\z@ \leftskip#3\relax \advance\leftskip\@tempdima\relax
    \rightskip\@pnumwidth plus4em \parfillskip-\@pnumwidth
    #5\leavevmode\hskip-\@tempdima
      \ifcase #1
       \or\or \hskip 1em \or \hskip 2em \else \hskip 3em \fi%
      #6\nobreak\relax
    \hfill\hbox to\@pnumwidth{\@tocpagenum{#7}}\par%
    \nobreak
    \endgroup
  \fi}
\makeatother
\usepackage{kmacros3}
\hypersetup{hidelinks}

\input{xy}
\xyoption{all}
\usepackage{tikz}
\usepackage{tikz-cd}
\usetikzlibrary{decorations.markings}
\usetikzlibrary{cd}
\usepackage{mabliautoref}
\usepackage{bm}

\usepackage{pifont}

\DeclareSymbolFont{rsfs}{OMS}{rsfs}{m}{n}
\DeclareSymbolFontAlphabet{\scr}{rsfs}

\newcommand{\cF}{{\mathcal F}}

\newcommand{\cQ}{{\mathcal Q}}

\renewcommand{\fram}{\mathfrak{m}}
\numberwithin{equation}{theorem}

\usepackage{fullpage}

\usepackage{setspace}

\usepackage{enumerate}
\usepackage{enumitem}
\usepackage{graphicx}
\usepackage[all,cmtip]{xy}

\usepackage{upgreek}

\usepackage{bbding}
\usepackage{verbatim}

\renewcommand{\sO}{\mathcal{O}}

\begin{document}

\title{On rational chain connectedness of globally $+$-regular varieties}
\author{Emre Alp \"Ozavc\i, Zsolt Patakfalvi, Kevin Tucker, Joe Waldron and Zheng Xu}

\address{\'Ecole Polytechnique F\'ed\'erale de Lausanne (EPFL), MA C3 615, Station 8, 1015 Lausanne, Switzerland}
\email{emre.ozavci@epfl.ch}
\address{\'Ecole Polytechnique F\'ed\'erale de Lausanne (EPFL), MA C3 635, Station 8, 1015 Lausanne, Switzerland}
\email{zsolt.patakfalvi@epfl.ch}
\address{Department of Mathematics, University of Illinois at Chicago, Chicago, IL 60607, USA}
\email{kftucker@uic.edu}
\address{Department of Mathematics, Michigan State University, East Lansing, MI 48824, USA}
\email{waldro51@msu.edu}
\address{Beijing International Center for Mathematical Research,
Peking University, No. 5 Yiheyuan Road, Haidian District, Beijing 100871, China}
\email{zhengxu@pku.edu.cn}

\begin{abstract}{We prove that globally $+$-regular pairs $(X,\Delta)$ for which $K_X+\Delta$ is $\mathbb{Q}$-Cartier are rationally chain connected in dimension three and mixed characteristic with residue characteristic $p>5$.  We also introduce a notion of strongly globally $+$-regular pairs and, under the same $\mathbb{Q}$-Cartier hypothesis, show that they are rationally chain connected in arbitrary dimension over a dense open subset of $\Spec(\mathbb{Z})$.}\end{abstract}
\maketitle

\section{Introduction}

The geometry of Fano varieties in characteristic zero is well understood, and in particular they are known to be rationally connected by work of Campana \cite{CampanaConnexite}  and Koll\'ar-Miyaoka-Mori \cite{KollarRationalConnecedness} in the smooth case and \cite{ZhangRational} in the mildly singular case. The original proofs do extend in the smooth case to prove that smooth Fano varieties are rationally chain connected over fields of positive characteristic. However, the singular versions are unknown in positive characteristic in dimensions higher than $3$. The reason is that all known proofs in the singular case use Kodaira or Kawamata-Viehweg type vanishing theorems at one point.  In positive characteristic, at least up to dimension $3$ or $4$, we often get around this using the Frobenius morphism  \cite{SmithFujitaFreenessForVeryAmple, HaconXuThreeDimensionalMinimalModel, CasciniTanakaXuOnBasePointFreeness}.  In particular, using Frobenius one can define a more well behaved class of varieties called globally $F$-regular varieties, which in particular are known to be Fano type by \cite{SchwedeSmithLogFanoVsGloballyFRegular}.  In mixed characteristic, the Frobenius morphism is not available, but recently a new class of varieties called globally $+$-regular varieties have been defined in \cite{BMPSTWW1}, which satisfy many of the same good properties as globally $F$-regular varieties.  This theory was used to develop the log minimal model program in dimension three and mixed characteristic with residue characteristic $p>5$.

However, it is unknown whether globally $+$-regular varieties are of Fano type, and this is a major open problem.  A solution would imply that globally $+$-regular varieties are rationally chain connected, using the characteristic zero results above.  In this paper, we  provide evidence for the conjecture by proving rational chain connectedness of certain classes of globally $+$-regular varieties.  See \cite{KV23} for other such evidence.   {Our first result, concerning threefolds over a mixed-characteristic DVR with residue characteristic $p>5$, is made possible by recent advances in the minimal model program for threefolds in mixed characteristic \cite{BMPSTWW1}.}

\begin{theorem*}
	{(\autoref{main_thm_dvr})} {Let $f:X\to V=\Spec(R)$ be a projective, surjective morphism {satisfying $f_*\sO_X=\sO_V$}, where $R$ is a DVR of mixed characteristic with residue characteristic $p>5$.  Suppose that $(X,\Delta)$ is a globally $+$-regular pair, that $K_X+\Delta$ is $\mathbb{Q}$-Cartier, and that $\dim X\leq 3$.}
    Then $X$ is rationally chain connected over $V$.
\end{theorem*}

Here being rationally chain connected over $V$ means that each fiber over $V$ is rationally chain connected, {which in this case is equivalent to} the generic fiber over $V$ being rationally chain connected (\autoref{lem:rational_chain_closed}). Note that the notion of absolute rational chain connectedness does not make sense in our situation, as the base $V$ is a mixed characteristic DVR and hence one cannot make sense of the notion of rational curve in the horizontal direction.
{Since globally $+$-regular pairs $(X,\Delta)$ for which $K_X+\Delta$ is $\mathbb{Q}$-Cartier have klt singularities \cite[Proposition 6.10]{BMPSTWW1},} it follows from \cite[Corollary 1.8]{HM07} that the generic fiber is rationally connected, {rather than merely rationally chain connected}.

While globally $+$-regular varieties have many of the same properties as globally $F$-regular varieties, they are not known to be equivalent even over a field of characteristic $p>0$, and it is unknown whether globally $+$-regular varieties are rationally chain connected even in this case.
{In this direction, we also obtain a positive-characteristic analogue in relative dimension at most three, generalizing \cite{gongyo_rational_2015}.}

\begin{theorem*}
{(\autoref{char_p_spread}) Let $(X,\Delta)$ be a globally $+$-regular pair {such that $K_X+\Delta$ is $\mathbb{Q}$-Cartier,} with flat projective morphism $g:X\to Z$ satisfying $g_*\sO_X=\sO_Z$, where $Z$ is an excellent scheme of characteristic $p\geq 11$ admitting a dualizing complex, and $\dim(X)-\dim(Z)\leq 3$.
Then $X$ is rationally chain connected over $Z$.  }
\end{theorem*}

Finally, we introduce a variation of globally $+$-regular varieties which we call strongly globally $+$-regular varieties.
This builds a perturbation into the definition, analogous to that which was used in the local case in \cite{BMPSTWW2} to unify the theory of mixed characteristic test ideals. {For quasi-projective varieties over an $F$-finite field of characteristic $p>0$, this class is equivalent to global $F$-regularity}, {while the result below gives rational chain connectedness in arbitrary dimension for the underlying varieties of pairs that are strongly globally $+$-regular over a dense open subset of $\Spec(\mathbb{Z})$.}

\begin{theorem*}{ (\autoref{main_thm_dense})
Let $(X,\Delta)$ be a strongly globally $+$-regular pair {such that $K_X+\Delta$ is $\mathbb{Q}$-Cartier,} with flat projective morphism $g:X\to Z$ satisfying $g_*\sO_X=\sO_Z$, {where $Z$ is an integral normal excellent scheme of mixed characteristic that admits a dualizing complex and whose closed points are dense and have positive residue characteristic}.
Then $X$ is rationally chain connected over $Z$.  }
\end{theorem*}

{Alternatively, assuming that the non-vanishing conjecture in characteristic zero holds up to the relative dimension (see \autoref{non-vanishing_conjecture}), \autoref{main_thm_dense} gives the same conclusion for globally $+$-regular pairs.  In particular, this gives the corresponding result in relative dimension three, where non-vanishing is known.}

\subsection*{Acknowledgements}
The authors would like to thank Bhargav Bhatt, Hanlin Cai, Linquan Ma, Karl Schwede, Zhiyu Tian and Jakub Witaszek for helpful conversations related to this work, and the referee for their helpful feedback.

\"Ozavcı and Patakfalvi were supported by the ERC starting grant \#804334 and by the Swiss National Science Foundation grant (FNS project grant) \#200021-
231484.

Tucker was supported in part by NSF Grants \#2200716, 2501904 and Simons Foundation Travel Support for Mathematicians SFI-MPS-TSM-00014083.

Waldron was supported by NSF CAREER Grant DMS \#2440240, NSF Grant DMS \#2401279 and the Simons Foundation Gift ID \#850684

\section{Preliminaries}

\subsection{Conventions}

All schemes we consider will be excellent with dualizing complex.   {A pair $(X,\Delta)$ consists of a normal integral scheme $X$ together with an effective $\mathbb{Q}$-divisor $\Delta$.  We do not assume in general that $K_X+\Delta$ is $\mathbb{Q}$-Cartier.  This hypothesis will be stated explicitly when needed and is included in our use of the terms klt and dlt.  For these and other singularity notions of birational geometry, see \cite{KollarMori} and \cite{KollarKovacsSingularitiesBook}.}

\subsection{Globally \texorpdfstring{$+$}{+}-regular varieties}

Here we recall the definition and basic properties of globally $+$-regular varieties.

\begin{definition}\cite[Definition 6.1]{BMPSTWW1}
 A pair $(X,\Delta)$ {where every closed point of $X$ has residue field of positive characteristic} is \emph{globally $+$-regular} if for every finite dominant map $f:Y\to X$ with $Y$ normal, the map $\sO_X\to f_*\sO_Y(\lfloor f^*\Delta\rfloor)$ splits as a map of $\sO_X$-modules.
\end{definition}

\begin{lemma}\cite[Lemma 5.1]{KV23}\label{lem:image_g+r}
    Let $X$ be a globally $+$-regular variety and $f:X\to Y$ a projective morphism such that $f_*\sO_X=\sO_Y$.  Then $Y$ is globally $+$-regular.
\end{lemma}

\begin{proposition}\cite[Proposition 6.19]{BMPSTWW1}
Let $(X,\Delta)$ be a globally $+$-regular pair. If $f : X \to Y$ is a proper birational morphism then  $(Y, f_*\Delta)$ is globally $+$-regular.
\end{proposition}

\begin{lemma}\label{prop:flip_G+R}
{Let $X$ and $Y$ be normal integral schemes projective over an affine excellent scheme admitting a dualizing complex and having only positive-characteristic closed points.} Suppose that $(X,\Delta)$ is a globally $+$-regular pair and that $\phi:X\dashrightarrow Y$ is a small birational map over the base, i.e. a rational map which is an isomorphism in codimension {one}. Then $(Y,\phi_*\Delta)$ is globally $+$-regular.
\end{lemma}

\begin{proof}
  Let $U$ be the largest big open set on which $X\dashrightarrow Y$ is an isomorphism.  Then for any normal finite cover $f: W\to X$, there is a splitting $\sO_X\to f_*\sO_W(\lfloor f^*\Delta\rfloor)\to \sO_X$.  Exactly as in \cite[Lemma 1.1.7]{BrionKumarFrobeniusSplitting}, this restricts to a splitting over $U$, and that splitting extends to a unique splitting for $W'\to Y$ where $W'$ is the normalization of $Y$ in the function field $K(W)$.  It follows that $(Y,{\phi_*\Delta})$ is globally $+$-regular since any normal finite cover of $Y$ arises from one of $X$ in this way.
\end{proof}

\begin{lemma}\label{lem:crepant_g+r}\cite[6.28]{BMPSTWW1}
{Let $f:Y\to X$ be a proper birational morphism between normal integral schemes, where $X$ is proper over an excellent affine scheme that admits a dualizing complex and whose closed points have positive residue characteristic.} Let $\Delta\geq 0$ be a $\mathbb{Q}$-divisor such that $(X,\Delta)$ is globally $+$-regular {and $K_X+\Delta$ is $\mathbb{Q}$-Cartier}.  Suppose we have $\Delta_Y\geq 0$ such that $K_Y+\Delta_Y=f^*(K_X+\Delta)$.  Then $(Y,\Delta_Y)$ is globally $+$-regular.
\end{lemma}

\subsection{Log minimal model program}

We recall some facts about the Log Minimal Model Program (LMMP).

\begin{theorem}\label{mmp}  Let $(X, B)$ be a projective klt pair over a field $k$ of characteristic $0$. If $K_{X}+B$ is not pseudo-effective, then we can run a terminating $(K_{X}+B)$-MMP which terminates with a Mori fibre space.
\begin{proof} By \cite[Subsection 2.3]{GNT19}, we can assume that $k$ is algebraically closed. Then the assertion follows from \cite[Corollary 1.3.2]{BCHM10}.
\end{proof}
\end{theorem}

\begin{proposition}\label{prop:small_Q-fact}
Let $Z$ be an integral normal Noetherian excellent scheme with a dualizing complex.  Assume that $Z$ is of mixed characteristic and has dense closed points.  Let $(X,B)$ be a klt pair which is projective and surjective over $Z$.  Then after replacing $Z$ with an open subset and $X$ with its preimage, there exists a {projective} small birational morphism $\pi:Y\to X$ such that $Y$ is $\mathbb{Q}$-factorial and $\pi^*(K_X+B)\sim_{\mathbb{Q}}K_Y+B_Y$ for a klt pair $(Y,B_Y)$.
\end{proposition}
\begin{proof}
    The existence of a small $\mathbb{Q}$-factorialization of the generic fiber is \cite[Corollary 1.4.3]{BCHM10}.  While the results of \cite{BCHM10} are stated over the complex numbers, it is known that the arguments and results descend to arbitrary fields of characteristic zero, see for example the survey article \cite{kawamata_finite_generation}.  More concretely, they are implied by the much stronger results of \cite{lyu_murayama}.

    Unfortunately the $\mathbb{Q}$-factorial condition does not automatically spread out from the generic fiber to our base $Z$, but instead we note that we can spread the construction from \cite{BCHM10}.
    First, by \cite{HironakaResolution} we may take a log resolution $f_{\eta}:W_{\eta}\to X_{\eta}$ of the generic fiber $X_{\eta}$ of $X\to Z$.  This spreads out to a morphism $f:W\to X$ after replacing $Z$ with some open subset $U\subset Z$ and replacing $X$ with the preimage of $U$.
Since the non-snc locus $V$ of $(W,B_W)$ is closed and does not meet the generic fiber $W_\eta$, its image $\phi_W(V)$ is closed in $Z$ by properness of $\phi_W : W \to Z$ and does not contain the generic point. Thus, after replacing $Z$ with $Z \setminus \phi_W(V)$, we may assume that $f : W \to X$ is a log resolution of $(X,B)$.

    As above, let $f:W\to X$ be this log resolution.  As in \cite{BCHM10} we construct a small $\mathbb{Q}$-factorialization by running a certain LMMP with scaling starting from this resolution.  In particular, let $B_W$ be the strict transform of $B$ and $E$ the reduced exceptional divisor of $f$.  Since $(X,B)$ is klt, we have
    $$K_W+B_W+E\sim_{\mathbb{Q}}f^*(K_X+B)+F$$
    where $F\geq0$ has support equal to the exceptional divisor of $f$.  If we can run a terminating $K_W+B_W+E$-MMP over $X$, resulting in $f_Y:Y\to X$, then since $K_W+B_W+E\sim_{f}F$ is nef over $X$, we see that $F_Y=0$ by the negativity lemma, and hence $f_Y$ is a small contraction.  Since MMP steps preserve $\mathbb{Q}$-factoriality, $Y$ is the required small $\mathbb{Q}$-factorialization.

    After restricting to the generic fiber we can run such a terminating MMP (e.g. with scaling of an ample divisor) by \cite{BCHM10}.    If $\phi_i:(Y_i)_\eta\to (V_i)_{\eta}$ is a contraction in this MMP on the generic fiber, we can spread out this step to a contraction $\phi_i:Y_i\to V_i$ after again shrinking the base $Z$.
    We claim that after further shrinking the base, it spreads out to a contraction of the LMMP.  The only condition which is not clear is that we can obtain $\rho(Y_i/V_i)=1$.  For this, let $D_1,...,D_m$ be a spanning set of $N^1(Y_i/V_i)$ consisting of divisors which are ample over $V_i$.
    Then since $(\phi_i)_{\eta}:(Y_i)_{\eta}\to (V_i)_{\eta}$ has Picard rank $1$, there are $\lambda_j>1$ such that $D_j|_{(Y_i)_{\eta}}\equiv_{({V_i})_{\eta}}\lambda_jD_1|_{(Y_i)_{\eta}}$.  Then by the base point free theorem \cite{KollarMori} we have $D_j|_{(Y_i)_\eta}\sim_{\mathbb{Q},{({V_i})_{\eta}}}\lambda_jD_1|_{(Y_i)_{\eta}}$.
    From this it follows that there is an open subset $U_j\subset Z$ such that $D_j|_{U_j}\sim_{\mathbb{Q},{({V_i})_{U_j}}}\lambda_jD_1|_{U_j}$.  Now we find that over $U=\cap_j U_j$, we have $\rho((Y_i)_U/(V_i)_U)=1$ as desired.  Therefore the contractions extend to contractions of relative Picard rank one over a restricted base.  It remains to prove that the flips extend.  But this is true by the same argument as \cite[Corollary 1.11]{das_waldron_imperfect} using the existence of good log minimal models in the relative birational case in characteristic zero, which follows from \cite{BCHM10} and the base point free theorem.
    At the end of the process, we obtain a morphism
     \(Y\to X\) from a $\mathbb{Q}$-factorial $Y$ which is small over the generic point of $Z$ and satisfies $\pi^*(K_{X_{\eta}}+B_{\eta})\sim_{\mathbb{Q}}K_{Y_{\eta}}+B_{Y_{\eta}}$.  After shrinking $Z$ further, we may assume that $\pi$ is small, and hence that $\pi^*(K_X+B)\sim_{\mathbb{Q}}K_Y+B_Y$ holds.
\end{proof}

Some of our results are contingent on the following conjecture:

\begin{conjecture}\label{non-vanishing_conjecture}\cite{KMM} (Non-vanishing Conjecture) Let $X$ be a smooth projective variety over $\C$. If $K_X$ is pseudo-effective, then $\kappa(X)\geq 0$, where $\kappa(X)$ denotes the Kodaira dimension of $X$.
\end{conjecture}

\noindent
{We will use this conjecture over arbitrary characteristic zero fields. This is equivalent to the
formulation over $\mathbb C$: the variety and the relevant divisor descend to a finitely generated subfield, pseudo-effectivity and Kodaira dimension are unchanged after extending scalars to an algebraically closed field, and any finitely generated field of
characteristic zero embeds into $\mathbb C$.
We will also use the following consequence. It is well-known that the non-vanishing conjecture for smooth projective varieties in dimension $d$ implies it holds for klt varieties in all dimensions up to $d$. More generally, by \cite[Theorem 1.4]{HashizumeNonVanishing}, the non-vanishing conjecture for smooth projective varieties with boundary zero
in dimension \(d\) implies the non-vanishing conjecture for projective log canonical
pairs in all dimensions at most \(d\). }

\subsection{Rational chain connectedness}
{
Here we recall the definitions and important properties of rationally (chain) connected varieties.

\begin{definition}\cite[IV.3.2]{KollarRationalCurves}\label{def:rationally_connected}
    A separated finite type scheme $X$ over a field $k$ is \emph{rationally chain connected} if there is
    \begin{itemize}
    \item a finite type scheme $Y$ over $k$, and
        \item  a proper family  $g \colon U\to Y$ of curves over $k$ together with a cycle morphism {$u \colon U\to X$}, for which
      the geometric fibers are connected and have rational irreducible components, and such that
    $$u^{(2)}\colon U\times_Y U\to X\times_{{ k}} X$$
    is dominant.
    \end{itemize}

    We say that $X$ is \emph{rationally connected} if there is a family as above whose geometric fibers are irreducible rational curves.
\end{definition}

\begin{remark}\label{rem:rcc_base_change}
    The above definition is stable under base change for a field extension $k \subseteq k'$ \cite[Exercise IV.3.2.5]{KollarRationalCurves} (see also \cite[Lemma 1.8]{gongyo_rational_2015} for a solution). However, it is sensitive to the chosen $k$-scheme structure. For example $\bP^1_{\bC}$ is rationally chain connected as a $\bC$-scheme, but not as an $\bR$-scheme.

    Note also that, for reducible schemes, alternate conventions exist in the literature (cf. \cite{HM07}) where disjoint unions of rationally chain connected varieties are considered rationally chain connected. However, our definition gives that rationally chain connected implies (geometrically) connected.
\end{remark}

\begin{definition}
    For a proper morphism $f \colon X\to S$, we say that $X$ is \emph{rationally chain connected over $S$} if $X_{\bar{s}}$ is rationally chain connected for every geometric point $\bar s \to S$.
    \end{definition}

The phrase ``rationally chain connected over \(S\)'' is not completely standard. Often, for a morphism over an integral base, one uses such language to refer only to the geometric generic fiber, or to general fibers. Here we need a fiberwise notion: our main results concern rational chain connectedness of every fiber, and in mixed characteristic there is no useful absolute notion of rational curves in the horizontal direction. Thus Definition~\ref{def:rationally_connected} is applied separately to the geometric fibers. The next lemma shows that, for proper flat families over an integral excellent base, this fiberwise condition can still be checked on the geometric generic fiber.

\begin{lemma}\cite[IV.3.5(2)]{KollarRationalCurves}\label{lem:rational_chain_closed}
Let \(f \colon X\to S\) be a proper flat morphism, with \(S\) integral and excellent. Assume that the geometric generic fiber \(X_{\bar\eta}\) is rationally chain connected. Then \(X_{\bar s}\) is rationally chain connected for every geometric point \(\bar s\to S\).
\end{lemma}

\begin{proof}
The argument in \cite[IV.3.5(2)]{KollarRationalCurves} is written under the assumption that the geometric fibers of \(f\) are irreducible. We record the modification needed when the fibers are allowed to be reducible.

We follow the notation and construction of \cite[IV.3.4--3.5]{KollarRationalCurves}. Since \(S\) is excellent, it is locally Noetherian. Since \(f\) is proper, it is of finite type, and therefore of finite presentation. Let \(R\subset \operatorname{Chow}_1(X/S)\) be the locus parametrizing geometrically connected one-cycles in the fibers of \(f\) whose geometric irreducible components are rational curves, and write \(R=\bigcup_i R_i\) for its irreducible components. Let \(g_i\colon U_i\to R_i\) be the universal families and put \(Y:=X\times_S X\). For each \(i\), let \(Z_i:=\operatorname{im}(U_i\times_{R_i}U_i\to Y)\), which is closed by properness. We give it the reduced induced structure, and all equalities involving the \(Z_i\)'s below are equalities of underlying closed subsets.

For a finite subset \(I\) of the index set, put \(Z_I:=\bigcup_{i\in I}Z_i\subseteq Y\), and define \[S_I:=\{s\in S\mid (Z_I)_{ s}=X_{ s}\times X_{ s} \mbox{ as closed subsets}\}.\] Since \(f\) is flat and of finite presentation, so is \(Y=X\times_S X\to S\); in particular, \(Y\to S\) is open by \cite[\href{https://stacks.math.columbia.edu/tag/01UA}{Lemma 01UA}]{stacks-project}. Thus, set-theoretically, \(S\setminus S_I=\operatorname{im}((Y\setminus Z_I)\to S)\) is open, and hence \(S_I\) is closed.

Now \(X_{\bar\eta}\) is rationally chain connected. By \cite[IV.3.4]{KollarRationalCurves}, the family in the definition may be represented on the Chow space: there is a finite type subscheme \(T\subset R_{\bar\eta}\) whose two-point evaluation map dominates \(X_{\bar\eta}\times X_{\bar\eta}\). The finitely many irreducible components of \(T\) are contained in the base changes of finitely many of the \(R_i\). Hence for some finite subset \(I\), the closed subset \((Z_I)_{\bar\eta}\) contains the dense image of this two-point evaluation map, and so \((Z_I)_{\bar\eta}=X_{\bar\eta}\times X_{\bar\eta}\). Thus, as equality of closed sets can be checked after passing to an algebraic closure, it follows that \(\eta\in S_I\). Since \(S_I\) is closed and \(S\) is integral, \(S_I=S\).

It follows that for every geometric point \(\bar s\to S\), the finite disjoint union of the base-changed families \((U_i)_{\bar s}\to (R_i)_{\bar s}\), for \(i\in I\), has geometric fibers that are connected and have rational irreducible components, and its two-point evaluation map to \(X_{\bar s}\times X_{\bar s}\) is dominant. Hence \(X_{\bar s}\) is rationally chain connected by Definition~\ref{def:rationally_connected}.
\end{proof}

\begin{remark}
The specialization result of the preceding lemma also appears in
\cite[Theorem~3.5(i)]{Starr2009ArithmeticFunctionFields}. Its flatness
hypothesis cannot in general be weakened to equidimensionality.
For example, if \(C\) is an integral nodal curve over an algebraically closed
field and \(\nu:\widetilde C\to C\) is its normalization, then
\(\widetilde C\times \mathbb P^1\to C\) is proper and equidimensional
with generic fiber \(\mathbb P^1\), but the fiber over the node is
\(\mathbb P^1\sqcup\mathbb P^1\), which is not rationally chain connected
under our convention. Equidimensionality can fail to suffice even when the
fibers are connected: let \(R\) be a DVR with residue field \(k\), and form
the pushout attaching an elliptic curve \(E/k\) as a vertical
component to the special fiber of \(\mathbb P^1_R\), identifying one
\(k\)-point of \(E\) with one \(k\)-point of \(\mathbb P^1_k\). The
resulting proper morphism to \(\operatorname{Spec}R\) has purely
one-dimensional fibers and rationally chain connected generic fiber, but
its special fiber is not rationally chain connected, since rational curves
cannot sweep out the elliptic component.

When \(f\) is projective, one can replace the Chow-space argument above with
one using Kontsevich spaces.
Rational chain connectedness of the geometric generic fiber gives finitely
many components of a suitable two-pointed genus-zero stable-map space whose
evaluation image dominates the generic fiber of \(X\times_S X\). Taking
their closures in the relative stable-map space, properness
\cite[Theorem~2.8 and \S2.5]{AbramovichOortStableMaps} gives a closed
evaluation image in \(X\times_S X\) containing its generic fiber. Since
\(X\times_S X\to S\) is flat, its generic fiber is dense, so this image is
all of \(X\times_S X\).
\end{remark}

{The following allows us to deduce rational connectedness of our generic fibers:}

\begin{theorem}\cite[Corollary 1.8]{HM07}
    Let $(X,\Delta)$ be dlt pair over a field of characteristic zero.  Then $X$ is rationally connected if and only if it is rationally chain connected.
\end{theorem}

See also \cite{KollarRationallyConnectedVarieties} for the smooth case. }

\section{Strongly globally \texorpdfstring{$+$}{+}-regular}

Allowing for perturbation along an effective divisor is an important enhancement of classes of singularities in all characteristics. For example, in characteristic zero, this roughly strengthens log canonical singularities to klt and Calabi-Yau to log Fano. In positive characteristic, perturbation of Frobenius splittings led to the development of the theory of tight closure, and upgrades globally F-split to globally F-regular varieties. In mixed characteristic, however, the situation is not yet well understood; it is hoped for but ultimately not yet known that $+$-regular and globally $+$-regular varieties are stable under small perturbation in analogy to their counterparts in equal characteristic. In the local setting, this has led to the competing notion of BCM-regular singularities, see \cite{MaSchwedeSingularitiesMixedCharBCM} and in particular \cite[Theorem C]{MaSchwedeSingularitiesMixedCharBCM}. Recently, \cite{BMPSTWW2} has allowed for the development of a unified theory of mixed characteristic test ideals up to such perturbations, see \autoref{BMPSTWW2theoremB} below. We now also introduce and explore a variant of globally $+$-regular that builds a perturbation into the definition. This generalization appeared implicitly in \cite[Setting 8.5, Theorem 8.25]{BMPSTWW1}, but given its relation with \cite{BMPSTWW2} and the stronger results obtained for it in this paper, we believe that it warrants a name.

\subsection{Definitions and basic properties}

{We assume that $X$ is normal and integral, that it is quasi-projective over an affine excellent scheme admitting a dualizing complex, and that both $X$ and the affine base have only positive-characteristic closed points.}

\begin{definition}
    We say that $(X,\Delta)$ is \emph{strongly globally $+$-regular} if for every effective divisor $D\geq 0$, there exists {$\epsilon\in\mathbb{Q}_{>0}$} such that $(X,\Delta+\epsilon D)$ is globally $+$-regular.
\end{definition}

\noindent {Every effective Weil divisor on $X$ is dominated by an effective Cartier divisor; consequently, it suffices to consider effective Cartier divisors $D$.}

The following from \cite{BMPSTWW2} shows that this recovers an important class of singularities when applied {locally.}

\begin{theorem}\cite[Theorem B]{BMPSTWW2}
\label{BMPSTWW2theoremB}
    {Let $V\to R$ be a local homomorphism, where $V$ is an excellent mixed-characteristic DVR and $R$ is a normal local domain that is flat and essentially of finite type over $V$. Let $\Delta\geq 0$ be a $\mathbb{Q}$-divisor on $X=\Spec(R)$} {such that $K_X+\Delta$ is $\mathbb{Q}$-Cartier}. Then the following are equivalent:
    \begin{enumerate}
        \item $(X,\Delta)$ is strongly globally $+$-regular,
         \item $\tau(X,\Delta)=\sO_X$ where $\tau$ is the test ideal from \cite{BMPSTWW2},
        \item $(X,\Delta)$ has perfectoid BCM-regular singularities \cite{MaSchwedeSingularitiesMixedCharBCM}.
        \end{enumerate}
\end{theorem}

We now prove some elementary properties of strongly globally $+$-regular varieties.
\begin{proposition}
Let $(X,\Delta)$ be an $F$-finite variety over a field of characteristic $p$.  Then $(X,\Delta)$ is globally $F$-regular if and only if it is strongly globally $+$-regular.
\end{proposition}

\begin{proof}
For a globally $F$-regular pair $(X,\Delta)$ and any effective divisor $D \geqslant 0$, the pair  $(X, \Delta + \epsilon D)$ is globally $F$-regular for $\epsilon > 0$ sufficiently small by \cite[Corollary 6.1]{SchwedeSmithLogFanoVsGloballyFRegular}. Then, $(X, \Delta + \epsilon D)$ is globally $+$-regular by \cite[Lemma 6.14]{BMPSTWW1} and $(X,\Delta)$ is strongly globally $+$-regular.

Now assume that $(X,\Delta)$ is strongly globally $+$-regular and let $D \geqslant 0$ be an effective divisor on $X$. The pair $(X, \Delta + \epsilon(D + \Delta))$ is globally $+$-regular for some $\epsilon > 0$. Let $e > 0$ be an integer such that $1/p^e < \epsilon$ so that $(X, \Delta + 1/p^e(D + \Delta))$ is globally $+$-regular. In particular, the twisted $e'$th iterate of the absolute Frobenius morphism
\[
    \sO_X \rightarrow F^{e'}_*\sO_X(\lfloor p^{e'}\Delta + p^{e' - e}(\Delta + D)\rfloor)
\]
is split. Notice that for $e' \gg 0$, we have $\lfloor p^{e'}\Delta + p^{e' - e}(\Delta + D)\rfloor \geqslant \lceil (p^{e'}-1)\Delta \rceil + D$ and $\sO_X \rightarrow F^{e'}_*\sO_X(\lceil (p^{e'}-1)\Delta \rceil + D)$ splits. Since $D$ is an arbitrary effective divisor, this shows that $(X,\Delta)$ is globally $F$-regular.
\end{proof}

\begin{proposition}\label{prop:image_strong_g+r}
{Let $X$ and $Y$ be normal integral schemes, where $X$ is quasi-projective over an affine excellent scheme admitting a dualizing complex and both $X$ and the affine base have only positive-characteristic closed points.} Let $\pi:Y\to X$ be a projective morphism satisfying $\pi_\ast\sO_Y=\sO_X$, and let $\Delta$ be an effective $\mathbb{Q}$-Cartier divisor on $X$. If $(Y,\pi^*\Delta)$ is strongly globally $+$-regular, then so is $(X,\Delta)$.
\end{proposition}
\begin{proof}
{Let $D \geqslant 0$ be an effective Cartier divisor on $X$.} Since $\pi$ is dominant the pullback $\pi^*D$ of $D$ is well-defined and effective. {Since $(Y,\pi^*\Delta)$ is strongly globally $+$-regular, $(Y,\pi^*(\Delta+\epsilon D))$ is globally $+$-regular for some $\epsilon>0$. Hence $(X,\Delta+\epsilon D)$ is globally $+$-regular by \cite[Lemma 5.1(iii)]{KV23}.}
\end{proof}

\begin{proposition}\label{prop:small_strong_g+r}
{Let $X$ and $Y$ be normal integral schemes quasi-projective over the same affine excellent scheme admitting a dualizing complex, and assume that $X$, $Y$, and the affine base have only positive-characteristic closed points. Suppose that $(X,\Delta)$ is strongly globally $+$-regular and that $f:X\dashrightarrow Y$ is a small birational map over the base.} Then $(Y,f_*\Delta)$ is strongly globally $+$-regular.
\end{proposition}

\begin{proof}
The proof is nearly identical to that of \autoref{prop:flip_G+R}, but we include the details for the sake of completeness. Let us show that for an open subset $U \subseteq X$ with ${\codim_X(X\setminus U) \geqslant 2}$, the pair $(U, \Delta|_{U})$ is strongly globally $+$-regular. Let $f: S \rightarrow U$ be a dominant finite cover of $U$ {with $S$ normal. Since $X$ is excellent, the normalization $\bar S$ of $X$ in $K(S)$ is finite over $X$ and restricts to $S$ over $U$. Thus, $f$ extends to a dominant finite cover $\bar f:\bar S\to X$.} {Fix an effective divisor $D$ on $X$, and choose $\epsilon>0$ such that $(X,\Delta+\epsilon D)$ is globally $+$-regular.} Then the natural map $\sO_X \rightarrow \bar{f}_*\sO_{\bar{S}}(\lfloor \bar{f}^* (\Delta  + \epsilon D)\rfloor)$ splits. Restricting this splitting to $U$, gives a splitting of $\sO_U \rightarrow f_*\sO_S(\lfloor f^* (\Delta|_U  + \epsilon D|_U)\rfloor)$. Since ${\codim_X(X\setminus U) \geqslant 2}$ and $X$ is normal, every divisor on $U$ uniquely extends to a divisor on $X$ and $(U, \Delta|_U)$ is strongly globally $+$-regular.

Conversely, suppose that there exists $U \subseteq X$ with ${\codim_X(X\setminus U) \geqslant 2}$, such that $(U, \Delta|_U)$ is a strongly globally $+$-regular pair.  We show that $(X, \Delta)$ is strongly globally $+$-regular. {Fix an effective divisor $D$ on $X$, and choose $\epsilon>0$ such that $(U,\Delta|_U+\epsilon D|_U)$ is globally $+$-regular.} Let $g: S \rightarrow X$ be a dominant finite cover of $X$; considering the normalization of $S$, which is finite dominant over $S$, we may assume that $S$ is normal. Since $S$ is normal and $\sO_S(\lfloor g^* (\Delta  + \epsilon D)\rfloor)$ is reflexive on $S$, by \cite[Proposition 5.7.9]{EGA_IV_II}, $g_*\sO_S(\lfloor g^* (\Delta  + \epsilon D)\rfloor)$ is reflexive as well. {The splitting over $U$ extends to a splitting of $\sO_X \rightarrow g_*\sO_S(\lfloor g^* (\Delta  + \epsilon D)\rfloor)$. Since $D$ and $g$ were arbitrary, $(X,\Delta)$ is strongly globally $+$-regular.}

Now, applying the first claim to the biggest open subset $V \subseteq X$ over which $f$ is an isomorphism and the second claim to $f(V) \subseteq Y$, shows that $(Y, f_*\Delta)$ is strongly globally $+$-regular.
\end{proof}

\begin{proposition}\label{prop:crepant_strong_g+r}
    Let $(X,\Delta)$ be a strongly globally $+$-regular pair {which is projective over an affine excellent scheme admitting a dualizing complex and having only positive-characteristic closed points}, {and assume that $K_X+\Delta$ is $\mathbb{Q}$-Cartier}. Let $\pi:Y\to X$ be a projective birational morphism.  Let $\pi^*(K_X+\Delta)=K_Y+\Delta_Y$ and suppose that $\Delta_Y\geq 0$. Then $(Y,\Delta_Y)$ is strongly globally $+$-regular.
\end{proposition}
\begin{proof}
    {Let $D\geqslant 0$ be an effective Cartier divisor on $Y$, and set $\mathcal I:=\pi_*\sO_Y(-D)\subseteq\pi_*\sO_Y=\sO_X$. If $H$ is an ample Cartier divisor on $X$, then $\mathcal I(nH)$ is globally generated for $n\gg 0$. A nonzero section of $\mathcal I(nH)\subseteq\sO_X(nH)$ defines an effective Cartier divisor $D'$ on $X$ such that $\pi^*D'\geqslant D$. Since $X$ is strongly globally $+$-regular, $(X,\Delta+\epsilon D')$ is globally $+$-regular for some $\epsilon>0$. Applying \cite[Proposition 6.28]{BMPSTWW1} shows that $(Y,\Delta_Y+\epsilon\pi^*D')$ is globally $+$-regular, which implies that $(Y,\Delta_Y+\epsilon D)$ is globally $+$-regular.}
\end{proof}

\subsection{Examples}

\begin{proposition}
Let $(X,\Delta)$ be projective over $\mathrm{Spec}(\mathbb{Z})$.  If $(X_\mathbb{Q},\Delta_\mathbb{Q})$ is log Fano then there exists an open subset $U\subseteq \Spec(\mathbb{Z})$ such that for each closed point $P\in U$, the localization of $(X,\Delta)$ at $P$ is strongly globally $+$-regular.  Furthermore, $(X_U,\Delta_U)$ is strongly globally $+$-regular over $U$.

\end{proposition}
\begin{proof}
{Denote the natural map $X\to\Spec(\mathbb{Z})$ by $g$.}
{By \cite[Theorem 1.2]{SchwedeSmithLogFanoVsGloballyFRegular}, there exists an open subset $U\subset \Spec(\mathbb{Z})$ such that, for every closed point $P\in U$, the induced pair $(X_P,\Delta_P)$ on the fiber over $P$ is globally $F$-regular. After shrinking $U$ further, {we may assume that $K_{X_U}+\Delta_U$ is $\mathbb{Q}$-Cartier and that $-(K_{X_U}+\Delta_U)$ is ample over $U$} and that the induced morphism $g_U:X_U\to U$ satisfies $(g_U)_*\sO_{X_U}=\sO_U$. We claim that this $U$ satisfies the conclusion of the proposition. Fix an effective Cartier divisor $D$ on $X_U$. Write $D=D'+\sum_{P\in T}m_PX_P$, where $D'$ has no vertical components and $T\subset U$ is finite; in particular, $D'$ is Cartier. Choose $\epsilon>0$ sufficiently small that {$-(K_{X_U}+\Delta_U+\epsilon D)$ is ample over $U$}, that $\epsilon m_P<1$ for every $P\in T$, and that $(X_P,\Delta_P+\epsilon D'_P)$ is globally $F$-regular for every $P\in U$. Indeed, \cite[Theorem 1.2]{SchwedeSmithLogFanoVsGloballyFRegular} gives the latter over an open subset $U'\subset U$, and, since $U\setminus U'$ is finite and $(X_P,\Delta_P)$ is globally $F$-regular at every remaining point, we may decrease $\epsilon$ so that it holds everywhere on $U$. Fix $P\in U$, let $\widetilde X$ be the localization at $P$, and set $m_P=0$ if $P\notin T$. Since $\Diff_{X_P}(\widetilde\Delta+\epsilon\widetilde{D'})=\Delta_P+\epsilon D'_P$, \cite[Corollary 7.5]{BMPSTWW1} shows that $(\widetilde X,X_P+\widetilde\Delta+\epsilon\widetilde{D'})$ is purely globally $+$-regular. As $\epsilon m_P<1$, \cite[Lemma 6.27]{BMPSTWW1} then shows that $(\widetilde X,\widetilde\Delta+\epsilon\widetilde{D'}+\epsilon m_PX_P)=(\widetilde X,\widetilde\Delta+\epsilon\widetilde D)$ is globally $+$-regular. Thus the localization of $(X,\Delta)$ at every $P\in U$ is strongly globally $+$-regular.}

To see the last part, {using the same $\epsilon$ as above}, note that for a finite cover $f : Y \rightarrow X$, the natural morphism $\sO_X \rightarrow f_*\sO_Y(\lfloor f^*(\Delta + \epsilon D)\rfloor)$ splits if and only if the evaluation map
\[
    \Hom(f_*\sO_Y(\lfloor f^*(\Delta + \epsilon D)\rfloor), \sO_X) \rightarrow \sO_X(X)
\]
is surjective.
Now, localizing the above map at a fixed $P \in U$, we get,
\begin{equation}\label{eq:locEvaluationMap}
    \Hom(\widetilde{f}_*\sO_{\widetilde{Y}}(\lfloor \widetilde{f}^*(\widetilde{\Delta} + \epsilon \widetilde{D})\rfloor), \sO_{\widetilde{X}}) \rightarrow \sO_{\widetilde{X}}(\widetilde{X}),
\end{equation}
where $\widetilde{Y}$ is the localization of $Y$ and $\widetilde{f} : \widetilde{Y} \rightarrow \widetilde{X}$ is the induced finite cover.
As {we showed earlier that $(\widetilde{X},\widetilde{\Delta}+\epsilon\widetilde{D})$ is} globally $+$-regular, the evaluation map \autoref{eq:locEvaluationMap} is surjective.  Therefore the sheaf map
\[{g_*\mathscr{H}om(f_*\sO_Y(\lfloor f^*(\Delta + \epsilon D)\rfloor), \sO_X)} \rightarrow g_*\sO_X\]
is surjective in some neighbourhood of $P$.
Since this holds for all $P$, and since $U$ is affine, this gives the required surjectivity.  This shows that $(X_U, \Delta_U)$ is strongly globally $+$-regular.
\end{proof}

Conjecturally, the converse of this proposition holds:
\begin{conjecture}
Let $(X,\Delta)$ be a strongly globally $+$-regular variety, projective over an open subset of $\mathrm{Spec}(\mathbb{Z})$.  Then $(X,\Delta)$ (and hence $(X_{\mathbb{Q}},\Delta_{\mathbb{Q}})$) is of log Fano type.
\end{conjecture}

However, the existence of the boundary that would make $(X,\Delta)$ log Fano is {a} major open question. Of course, one should ask the same question for globally $+$-regular varieties as well \cite[Conjecture 6.17]{BMPSTWW1}, as well as whether the addition of a perturbation into the definition of globally $+$-regular varieties is ultimately needed.

\begin{question}
    {If $(X,\Delta)$ is globally $+$-regular, is it necessarily strongly globally $+$-regular?}
\end{question}

\section{Rationally chain connectedness of globally \texorpdfstring{$+$}{+}-regular varieties in mixed characteristic}

{We first show that, under a mild assumption that the closed points of the base are dense,
the underlying varieties of strongly globally $+$-regular pairs $(X,\Delta)$ for which $K_X+\Delta$ is $\mathbb{Q}$-Cartier are rationally chain connected in mixed characteristic.  The non-vanishing conjecture over $\C$ gives the same conclusion for globally $+$-regular pairs $(X,\Delta)$ for which $K_X+\Delta$ is $\mathbb{Q}$-Cartier} (see Theorem \ref{main_thm_dense}).

\begin{setting}\label{setting} We fix $Z$ to be an integral normal excellent scheme with a dualizing complex. Assume that $Z$ is of mixed characteristic and {that its closed points are dense and have positive residue characteristic}.
\end{setting}

We need two lemmas before we prove the main result of this section.

\begin{lemma}\label{top_cohomology_vanish} Let $g:X\to Z$ be a surjective and projective morphism {of relative dimension $d>0$, with $Z$ as in \autoref{setting}, and let $\eta$ be the generic point of $Z$.  Let $\cF$ be a $\Q$-Cartier Weil divisorial sheaf on $X$.}
Write $X_\eta$ for the generic fibre of $g$, and $\cF_\eta$ for the base change of $\cF$ to $X_\eta$.

Assume one of the following pairs of conditions hold:
\begin{enumerate}
\item[(1)]  $X$ is strongly globally $+$-regular, and  $\cF_\eta$ is pseudo-effective.

\item[(2)]  $X$ is globally $+$-regular, and {$\kappa(X_{\eta},\cF_{\eta})\geq 0$.}
\end{enumerate}

Then $H^d(X_\eta,\cF_\eta)=0$.

\begin{proof}
 If we are in (1), after {replacing $Z$ by a nonempty affine open} if necessary, we can take $A$ to be an effective ample divisor on $X$. Then there exists a rational $\epsilon>0$ such that $(X,\epsilon A)$ is globally $+$-regular. We write $\cF\sim \cO_X(G)$, where $G$ is a $\Q$-Cartier Weil divisor. {Since $\cF_\eta$ is pseudo-effective and $A_\eta$ is ample, the $\Q$-Cartier $\Q$-divisor $G_\eta+\epsilon A_\eta$ is big and hence $\Q$-linearly equivalent to an effective $\Q$-divisor. After shrinking $Z$ further, this $\Q$-linear equivalence spreads out, so $G+\epsilon A$ is $\Q$-linearly equivalent to an effective $\Q$-Cartier $\Q$-divisor.}
By the covering trick (see \cite{tomari_watanabe} or \cite[Lemma 4.5]{BlickleSchwedeTuckerTestAlterations}), we can take a finite cover $f_1:Y_1\to X$ such that $Y_1$ is normal, $f_1^\ast(\epsilon A)$ is integral and
$$\big(f_1^\ast\cF\otimes\cO_{Y_1}(f_1^\ast(\epsilon A))\big)^{\ast\ast}\sim \cO_{Y_1}(D)$$
where $D$ is an effective Cartier divisor on $Y_1$.

If we are in (2), {then $\kappa(X_\eta,\cF_\eta)\geq 0$ implies that some reflexive power $\cF_\eta^{[m]}$ has a nonzero section. After shrinking $Z$, this section extends to a nonzero section of $\cF^{[m]}$.} By the covering trick, we can take a finite cover $f_1:Y_1\to X$ such that $Y_1$ is normal and $(f_1^{\ast}\cF)^{**}\sim \cO_{Y_1}(D)$, where $D$ is an effective Cartier divisor on $Y_1$.

Write $g_1:=g\circ f_1$.
Since $Z$ has dense closed points, we can choose $z\in Z$ to be a {closed point of positive residue characteristic} such that $D$ doesn't contain any component of $g_1^{-1}(z)$. By flat base change for cohomology, we can replace $Z$, $X$ and $Y_1$ by their localizations at $z$.
Moreover, we can assume that $\cO_X,\cF,\cO_{Y_1}$ are flat over $Z$, and $R^dg_\ast\cO_X, R^dg_\ast\cF, R^dg_{1,\ast}\cO_{Y_1}$ are locally free. We denote the residue characteristic of $Z$ by $p$, and denote the restriction to the fibre over $z$ by $\cdot_{z=0}$. Then we have the exact sequences
$$0\to \cO_{Y_1,z=0}\to \cO_{Y_1,z=0}(D)\to \cQ:=\cO_{Y_1,z=0}(D)/\cO_{Y_1,z=0}\to 0,$$
and
$$H^d(\cO_{Y_1,z=0})\to H^d(\cO_{Y_1,z=0}(D))\to H^d(\cQ).$$
Note that $\cQ$ is supported on a closed subscheme of $Y_{1,z=0}$ with the dimension $\leq d-1$. Hence $H^d(Y_{1,z=0},\cQ)=0$. It follows that $H^d(\cO_{Y_{1,z=0}})\to H^d(\cO_{Y_{1,z=0}}(D))$ is a surjection. By killing cohomology \cite[Proposition 3.1 (a)]{BMPSTWW1}, we can take a finite cover $f_2:Y_2\to Y_1$ such that $Y_2$ is normal and the pullback $f_2^\ast:H^d(\cO_{Y_1,p=0})\to H^d(\cO_{Y_2,p=0})$ is the zero map.
Since $g_1$ is flat and $R^dg_{1,\ast}\cO_{Y_1}$ is locally free, the pullback $$f_2^\ast:H^d(\cO_{Y_1,z=0})\cong H^d(\cO_{Y_1,p=0})_{z=0}\to H^d(\cO_{Y_2,p=0})_{z=0} \to H^d(\cO_{Y_2,z=0})$$ is the zero map. Note that we have surjections
$$H^d(\cO_{Y_{1,z=0}})\to H^d(\cO_{Y_{1,z=0}}(D))$$ and $$H^d(\cO_{Y_{2,z=0}})\to H^d(\cO_{Y_{2,z=0}}({f_2^*}D)).$$
It follows that
$$H^d(\cO_{Y_{1,z=0}}(D))\to H^d(\cO_{Y_{2,z=0}}({f_2^*}D))$$
is the zero map.
We write $f:=f_1\circ f_2$. Then the natural map
$$f^\ast :H^d(X_{z=0},\cF_{z=0})\to H^d(\cO_{Y_{2,z=0}}({f_2^*} D))$$
is the zero map.

If we are in (1), since $(X,\epsilon A)$ is globally $+$-regular, we have a splitting map
$\cO_X\to f_\ast\cO_{Y_2}(f^{\ast}(\epsilon A))$. Tensoring it with $\cF$ and taking the reflexivizations, we get a splitting map
$\cF\to f_\ast\cO_{Y_{2}}({f_2^*}D).$
If we are in (2), since $X$ is globally $+$-regular, we have a splitting map
$\cO_X\to f_\ast\cO_{Y_2}$. Tensoring it with $\cF$ and taking the reflexivizations, we get a splitting map $\cF\to f_\ast\cO_{Y_{2}}({f_2^*}D).$

It gives a splitting map
 $$\cF_{z=0}\to (f_\ast\cO_{Y_{2}}({f_2^*}D))_{z=0}\cong f_\ast( \cO_{Y_{2,z=0}}({f_2^*}D)).$$

 Thus $H^d(X_{z=0},\cF_{z=0})=0$. Note that $H^d(X,\cF)$ is a finitely generated free $\cO_Z$-module and $$H^d(X,\cF)_{z=0}= H^d(X_{z=0},\cF_{z=0})=0.$$
 Hence $H^d(X,\cF)=0$ by Nakayama's lemma. By flat base change for cohomology, we have $H^d(X_\eta,\cF_\eta)=0$.
\end{proof}
\end{lemma}

\begin{lemma}\label{uniruled}Let $g:X\to Z$ be a surjective projective morphism, with $Z$ as in \autoref{setting}, {$X$ is globally $+$-regular over $Z$ and $K_X$ is $\mathbb{Q}$-Cartier}. Write $X_\eta$ for the generic fibre of $g$, where $\eta$ is the generic point of $Z$, {and assume that $d:=\dim X_\eta > 0$.} Then $\kappa(X_\eta)=-\infty$.

Moreover, assume that one of the following conditions hold.
\begin{enumerate}
\item[(1)] $X$ is strongly globally $+$-regular scheme over $Z$.
\item[(2)] Conjecture \ref{non-vanishing_conjecture} holds in dimension {$d$}.
\end{enumerate}
Then $K_{X_\eta}$ is not pseudo-effective. In particular, $X_\eta$ is uniruled.
\begin{proof}   {By \cite[Proposition 6.10]{BMPSTWW1}, $X$ is klt, and hence $X_\eta$ is klt and Cohen--Macaulay.} If $\kappa(X_\eta)\geq 0$, by Lemma \ref{top_cohomology_vanish}, we get
$$H^0(X_\eta,\cO_{X_\eta})= H^d(X_\eta,K_{X_\eta})^\vee=0,$$
which is impossible. Hence we have $\kappa(X_{\eta})=-\infty$.

To show that $K_{X_\eta}$ is not pseudo-effective, in Case (1) \autoref{top_cohomology_vanish} leads to a contradiction as above, and in Case (2)  non-vanishing and $\kappa(X_\eta)=-\infty$ implies that $K_{X_\eta}$ is not pseudo-effective.
By Theorem \ref{mmp}, we can run a $K_{X_\eta}$-MMP, which terminates with a Mori fibre space $\pi_\eta:X_\eta'\to B_\eta$ since $K_{X_\eta}$ is not pseudo-effective. Therefore, $X_\eta$ is uniruled.
\end{proof}
\end{lemma}

\begin{theorem}\label{main_thm_dense} Let $g:X\to Z$ be a {flat} projective morphism such that $g_*\sO_X=\sO_Z$, with $Z$ as in \autoref{setting}, {and let $(X,B)$ be a pair such that $K_X+B$ is $\mathbb{Q}$-Cartier}. Let $\eta$ be the generic point of $Z$. Assume that one of the following conditions hold.
\begin{enumerate}
    \item[(1)] $(X,{B})$ is a strongly globally $+$-regular pair over $Z$.
\item[(2)]$(X,{B})$ is a globally $+$-regular pair over $Z$, and Conjecture \ref{non-vanishing_conjecture} holds in { all dimensions $d\leq\mathrm{dim}\ X_\eta$. }
\end{enumerate}
Then $X$ is rationally chain connected over $Z$.
\end{theorem}
\begin{proof}
{By \autoref{lem:rational_chain_closed} and the fact that rational chain connectedness can be checked after field extension, it is sufficient to show that $X_\eta$ is rationally chain connected. {Replacing $Z$ by a nonempty affine open, we may assume that $Z$ is affine.} Set $d:=\dim X_\eta$. If $d=0$, then $g_*\sO_X=\sO_Z$ gives $H^0(X_\eta,\sO_{X_\eta})=K(Z)$. Since $X_\eta$ is integral and zero-dimensional, it follows that $X_\eta=\operatorname{Spec}K(Z)$, and hence the result follows. Thus, we may assume that $d>0$.}

{In either case, $(X,B)$ is globally $+$-regular, so \cite[Proposition 6.10]{BMPSTWW1} shows that $(X,B)$ is klt.} By \autoref{prop:small_Q-fact}, we can find a crepant small $\mathbb{Q}$-factorialization $\pi:Y\to X$ (after possibly shrinking $Z$). {Writing $\pi^*(K_X+B)=K_Y+B_Y$, the pair $(Y,B_Y)$ is strongly globally $+$-regular by \autoref{prop:crepant_strong_g+r} in Case (1), and globally $+$-regular by \autoref{lem:crepant_g+r} in Case (2). In particular, $Y$ is strongly globally $+$-regular in Case (1) and globally $+$-regular in Case (2). Moreover, since $Y$ is $\mathbb{Q}$-factorial, $K_Y$ is $\mathbb{Q}$-Cartier. As rational curves on $Y_\eta$ push forward to rational curves or points on $X_\eta$, it suffices to show that $Y_\eta$ is rationally chain connected. Thus, after replacing $X$ by $Y$, we may assume that $B=0$ and $K_X$ is $\mathbb{Q}$-Cartier.}

By Lemma \ref{uniruled}, $K_{X_\eta}$ is not pseudo-effective. {Since $X$ is klt, so is $X_\eta$.}
By Theorem \ref{mmp}, we can run a $K_{X_\eta}$-MMP over $K(Z)$, which terminates with a Mori fibre space $\pi_\eta:X_\eta'\to V_\eta$, where $V_\eta$ is a variety over $K(Z)$ satisfying $\dim(V_{\eta})<\dim(X_{\eta})$.  By the same spreading-out argument as in Proposition~\ref{prop:small_Q-fact}, after shrinking $Z$, we may assume that each step of this generic-fiber MMP is induced by a birational map over $Z$. {We may moreover assume that all spread-out models and contraction targets, including $X'$ and $V$, are normal, integral, and projective over $Z$.} Thus the final Mori fibre space also spreads out: there is a morphism $\pi:X'\to V$ over $Z$ whose restriction to the generic fiber is $\pi_\eta:X_\eta'\to V_\eta$. {Since $V_\eta$ is $\mathbb{Q}$-factorial by \cite[Corollary 3.18]{KollarMori},} {after shrinking $Z$ further, we may also assume that $q:V\to Z$ is flat, that $K_V$ is $\Q$-Cartier, and that $q_*\sO_V=\sO_Z$.}

Note that in Case (1) (resp. Case (2)),  $X'$ is still strongly globally $+$-regular by \autoref{prop:image_strong_g+r} and \autoref{prop:small_strong_g+r}
(resp. globally $+$-regular by \autoref{lem:image_g+r} and \autoref{prop:flip_G+R}), and hence by \autoref{prop:image_strong_g+r} (resp. \autoref{lem:image_g+r}) $V$ is also strongly globally $+$-regular (resp. globally $+$-regular).
{Since $\dim(V_\eta)<\dim(X_\eta')$, induction on the dimension of the generic fiber applies to $V\to Z$, and therefore $V_{\eta}$ is rationally chain connected over $\eta$.}

{Let $K$ be an uncountable algebraically closed extension of $k(\eta)$.  Since rational chain connectedness can be checked after field extension, it is enough to show that $X_K'$ is rationally chain connected.  For two very general closed points $x,y\in X_{K}'$, there is a chain $\{D_i\}$ of rational curves connecting $\pi_{K}(x)$ to $\pi_{K}(y)$ by rational chain connectedness of $V_K$ using \cite[Proposition 3.6]{KollarRationalCurves}.  Then by \cite[Corollary 1.10]{HM07}, for each $i$, there is a rational curve $C_i$ in $X_{K}'$ which surjects to $D_i$.  By \cite[Corollary 1.4]{HM07}, we see that there is a chain of rational curves connecting any two points in the same fiber of $\pi_K$.  Hence we conclude that $x$ and $y$ can be connected by a chain of rational curves.  Thus $X_{K}'$, and so $X_\eta'$, is rationally chain connected.  }{Since $X_\eta'$ is klt, \cite[Corollary 1.8]{HM07} shows that it is rationally connected. By the birational invariance of rational connectedness \cite[IV.3.3]{KollarRationalCurves}, $X_\eta$ is therefore rationally connected, and hence rationally chain connected.}
\end{proof}

\begin{corollary}
    {Let $g:X\to Z$ be a {flat} projective morphism {satisfying $g_*\sO_X=\sO_Z$}, where $Z$ is as in \autoref{setting}.  Assume that $X$ is globally $+$-regular over $Z$, that $K_X$ is $\mathbb{Q}$-Cartier, and that $\dim X\leq 4$.}  Then $X$ is rationally chain connected over $Z$.
\end{corollary}

It is clear that the same argument implies that if the following conjecture holds, then one can drop the assumption of non-vanishing conjecture in (2) of Theorem \ref{main_thm_dense}.

\begin{conjecture} Let $X$ be a klt projective variety over $\C$. If $H^i(X,L)=0$ for $i>0$ and any semi-ample line bundle on $X$, then $K_X$ is not pseudo-effective.
\end{conjecture}

{Using similar arguments and the mixed-characteristic LMMP, we prove rational chain connectedness in dimension at most three when the localization of $(X,\Delta)$ at a single fiber is globally $+$-regular, provided that the residue characteristic is $p>5$.}

\begin{theorem}\label{main_thm_dvr}
	{Let $f:X\to V=\Spec(R)$ be a projective, surjective morphism {satisfying $f_*\sO_X=\sO_V$}, where $R$ is a DVR of mixed characteristic with residue characteristic $p>5$.  Suppose that $(X,\Delta)$ is a globally $+$-regular pair, that $K_X+\Delta$ is $\mathbb{Q}$-Cartier, and that $\dim X\leq 3$.}
    Then $X$ is rationally chain connected over $V$.
\end{theorem}
\begin{proof}
   {If $\dim(X)=1$, then $X\to V$ is finite, and $H^0(X,\sO_X)=R$ implies that $X=V$, so the result is clear. Thus, we may assume that $\dim(X)\geq 2$.} {Since $K_X+\Delta$ is $\mathbb{Q}$-Cartier, \cite[Proposition 6.10]{BMPSTWW1} shows that $(X,\Delta)$ is klt.} We apply induction on $\dim(X)$.  By \cite[Lemma 6.7]{BMPSTWW1}, $X$ is globally $+$-regular.  Furthermore, we may assume that it is $\mathbb{Q}$-factorial by taking a small $\mathbb{Q}$-factorialization \cite[Corollary 9.21]{BMPSTWW1} and applying \autoref{prop:flip_G+R}. {After this replacement, $K_X$ is $\mathbb{Q}$-Cartier and $X$ is klt.}
    By \cite{BMPSTWW1} we can run a terminating $K_X$-MMP.  The outcome $Y$ of the MMP is globally $+$-regular by \autoref{prop:flip_G+R}  and \autoref{lem:image_g+r}.

    We claim that $Y$ admits the structure of Mori fiber space $\pi:Y\to Z$. Otherwise,  by \cite{bernasconi2024abundance}, $K_Y$ is semiample.
    By \cite[Lemma 4.5]{BlickleSchwedeTuckerTestAlterations}, there exists a finite cover $f:W\to Y$ such that $(f^*\omega_Y)^{**}$ is a semi-ample line bundle. Moreover, by \cite[Proposition 3.1 (a)]{BMPSTWW1}, replacing $f$ by a further cover, we may assume that
    $$f^*:H^i(Y_{p=0},\omega_Y)\to H^i(W_{p=0},(f^*\omega_Y)^{**})$$
    is a zero map for $i>0$. Since $Y$ is globally $+$-regular, we have that $f^*$ splits, and hence $H^i(Y_{p=0},\omega_Y)=0$ for $i>0$. Now we consider the exact sequence
    $$H^i(Y,\omega_Y)\overset{\cdot p}\to H^i(Y,\omega_Y)\to H^i(Y_{p=0},\omega_Y)=0.$$
    It implies that $H^i(Y,\omega_Y)$ is $p$-divisible for $i>0$. However, $H^i(Y,\omega_Y)$ is a finitely generated $R$-module. It follows that $H^i(Y,\omega_Y)=0$ for $i>0$, which is impossible by Serre duality and letting $i$ be $\mathrm{dim}\ Y-1$. Hence, the claim holds.

    Let $\xi$ be the generic point of $V$.  Note that $Z$ is globally $+$-regular by \autoref{lem:image_g+r}. {Since $Y$ is $\mathbb{Q}$-factorial and $\pi$ is a Fano contraction, the proof of \cite[Corollary 3.18]{KollarMori} applies here using the mixed-characteristic relative base-point-free theorem \cite[Theorem 9.27]{BMPSTWW1}, and shows that $Z$ is $\mathbb{Q}$-factorial.  In particular, $K_Z$ is $\mathbb{Q}$-Cartier.}
    Since $\dim(Z)<\dim(X)$,  we may conclude that $Z$ is  rationally chain connected over $V$ by induction on dimension.
    {We must have $\dim(Z)\in\{1,2\}$.  First, if $\dim(Z)=1$, so $Z=V$, then $Y_{\overline{\xi}}$ is rationally chain connected by \cite[Corollary 1.5]{HM07}.  If $\dim(Z)=2$, then $Z_{\overline{\xi}}\cong\mathbb{P}^1$.  $Y_{\overline{\xi}}\to Z_{\overline{\xi}}$ has a section by \cite[Corollary 1.10]{HM07} and geometrically rational fibers since they are Fano varieties of dimension $1$.  Hence $Y_{\overline{\xi}}$ is rationally chain connected by \cite[Proposition IV.3.6]{KollarRationalCurves}. Since $X_{\overline{\xi}}$ and $Y_{\overline{\xi}}$ are birational klt varieties in characteristic zero, \cite[Corollary 1.8]{HM07} and birational invariance of rational connectedness imply that $X_{\overline{\xi}}$ is rationally chain connected. Finally, $X\to V$ is flat, so \autoref{lem:rational_chain_closed} completes the proof.}
    \end{proof}

\section{Rational connectedness of globally \texorpdfstring{$+$}{+}-regular varieties in positive characteristic}

It was shown in \cite{gongyo_rational_2015} that projective globally $F$-regular varieties of dimension three over a field of positive characteristic are rationally chain connected. {Adapting their proof, we obtain the following result for globally $+$-regular pairs.}
\begin{proposition}\label{prop:char_p}
    {Let $k$ be a field of characteristic $p\geq 11$, and let $(X,\Delta)$ be a projective {globally} $+$-regular pair over $k$.  Assume that {$H^0(X,\sO_X)=k$}, that $K_X+\Delta$ is $\mathbb{Q}$-Cartier, and that $\dim X {\leqslant}3$.}  Then $X$ is rationally chain connected over $k$.
\end{proposition}

\begin{proof}
   {Since $K_X+\Delta$ is $\mathbb{Q}$-Cartier, \cite[Proposition 6.10]{BMPSTWW1} shows that $(X,\Delta)$ is klt.}
   {If $\dim X=0$, then $X=\Spec(k)$ and the conclusion is immediate. Thus we may assume that $\dim X>0$.} {By \cite[Theorem 9.1]{KV23}, the Kodaira dimension of $X$ is $\kappa(X) = -\infty$.} {If $\dim X = 1$, taking into account that globally $+$-regular varieties are normal, this shows that $X \cong \P^1_k$. If $\dim X=2$,
then \cite[Proposition 11.1]{KV23} shows that $X$ is rational,
and hence rationally chain connected. It remains to show the case $\dim X = 3$ which we consider for the rest of the proof.}

   {As $\dim X=3$, by \cite[Theorem 1.7]{Birkar16}, there is a projective birational morphism $f:Z\to X$ and an effective $\Q$-divisor $\Delta_Z$ such that $Z$ is $\Q$-factorial, $(Z,\Delta_Z)$ and hence $Z$ are terminal, and $K_Z+\Delta_Z=f^*(K_X+\Delta)$. By \autoref{lem:crepant_g+r}, $(Z,\Delta_Z)$ and hence $Z$ are globally $+$-regular. Since rational chains on $Z$ push forward to rational chains on $X$, it suffices to prove that $Z$ is rationally chain connected. We may therefore replace $X$ by $Z$ and assume that $X$ is $\Q$-factorial and terminal.}

Since $\kappa(X) = -\infty$ {by \cite[Theorem 9.1]{KV23}}, $K_X$ is not pseudo-effective by \cite[Theorem 1.1]{XuLeiNonvanishing}. Running a terminating $K_X$-MMP, see for instance \cite[Theorem 3.1]{gongyo_rational_2015}, we obtain a Mori fiber space structure $\pi: X' \rightarrow Y$ where $X'$ is $\Q$-factorial terminal. Note that both $X'$ and $Y$ are globally $+$-regular. The discussion included in Step 2 of the proof of \cite[Proposition 3.6]{gongyo_rational_2015} shows that $X$ is rationally chain connected if so is $X'$. Therefore we may assume that $X$ admits a Mori fiber space structure $f : X \rightarrow Y$. To conclude the proof, we argue depending on the dimension of $Y$. \newline \\
    \textbf{Case 1.} Assume $\dim (Y) = 0$.  {Then $X$ is a klt Fano threefold, so it is rationally chain connected by \cite[Proposition 4.11]{GNT19}.} \newline \\
    \textbf{Case 2.} Assume $\dim(Y) = 1$. {In this case, $Y$ is a globally $+$-regular curve over an algebraically closed field, so it is isomorphic to $\mathbb{P}^1$.} It follows from \cite[Corollary 1.7]{patakfalvisingularities} that $X$ is separably rationally connected. \newline \\
    \textbf{Case 3.} Assume $\dim(Y) = 2$. Let $K$ be the residue field of the generic point $\eta$ of $Y$ and note that by \cite[Lemma 4.3]{KV23}, the generic fiber $X_{\eta}$ of $f$ is globally $+$-regular. As before $X_{\bar{\eta}}$, the base change of $X_{\eta}$ to $\overline{K}$, is globally $+$-regular and in particular it is normal. As $X_{\bar{\eta}}$ is one dimensional over $\Spec(\overline{K})$ this is equivalent to $X_{\bar{\eta}}$ being regular. Since $-K_X$ is $f$-ample, the {geometric} generic fiber is a projective line. As above, $Y$ is a globally $+$-regular surface so it is rational, in particular it is separably rationally connected. Now take a resolution of $X' \rightarrow X$ that is an isomorphism outside of the singular locus of $X$. Since threefold terminal singularities are isolated, see \cite[Corollary 2.30]{KollarKovacsSingularitiesBook}, the geometric generic fibers of $X' \rightarrow Y$ and $f$ agree.
    Now, to apply \cite[Proposition 4.8]{gongyo_rational_2015}, take a resolution $r :Y' \rightarrow Y $ of $Y$ and pullback $X' \rightarrow Y$ along $r$ to get a morphism $Z \rightarrow Y'$. Let $Z'$ be an irreducible component of $Z$ dominating $X'$. Take a resolution $X'' \rightarrow Z'$ which is an isomorphism over the non-singular locus. We obtain the following diagram,
    \[\begin{tikzcd}
	{X''} & {X'} \\
	{Y'} & Y
	\arrow[from=1-1, to=1-2]
	\arrow[from=1-1, to=2-1]
	\arrow[from=1-2, to=2-2]
	\arrow[from=2-1, to=2-2]
    \end{tikzcd}\]
    where the horizontal arrows are birational morphisms and $X''$ and $Y'$ smooth. Note that since the generic fiber of $X' \rightarrow Y$ is smooth, the generic fibers of $X'' \rightarrow Y'$ and $X' \rightarrow Y$ coincide.  {Now since $Y'$ is also separably rationally connected}, we can apply
     \cite[Proposition 4.8]{gongyo_rational_2015} to $X'' \rightarrow Y'$ show that $X''$ is separably rationally connected which implies that $X'$, and thus $X$, is separably rationally connected.
 \end{proof}
\begin{remark}
    The above proof shows that, when $\dim X = 3$, if the base of the fibration $f: X \rightarrow Y$ is positive dimensional then $X$ is separably rationally connected. Moreover, the proof works over a field of characteristic $p > 5$, except for the second case where the base is one dimensional. In this case, the result we cite from \cite{patakfalvisingularities} applies for $p \geqslant 11$.
\end{remark}

\begin{corollary}\label{char_p_spread}
Let $(X,\Delta)$ be a globally $+$-regular pair such that $K_X+\Delta$ is $\mathbb{Q}$-Cartier, with flat projective morphism $g:X\to Z$ satisfying $g_*\sO_X=\sO_Z$, where $Z$ is an excellent scheme of characteristic $p\geq 11$ admitting a dualizing complex, and $\dim(X)-\dim(Z)\leq 3$.
Then $X$ is rationally chain connected over $Z$.
\end{corollary}

\begin{proof}
   This follows from \autoref{lem:rational_chain_closed} and \autoref{prop:char_p}.
\end{proof}

\bibliographystyle{amsalpha}
\bibliography{MainBib.bib}

\end{document}